\documentclass[graybox]{svmult}

\usepackage{mathptmx}       
\usepackage{helvet}         
\usepackage{courier}        
\usepackage{type1cm}        
\usepackage{makeidx}         
\usepackage{graphicx}        
\usepackage{multicol}        
\usepackage[bottom]{footmisc}

\usepackage{amssymb}
\usepackage{amsmath}
\usepackage{enumerate}
\usepackage{bm,nicefrac}
\usepackage{color}
\usepackage{comment}
\usepackage[matrix,arrow,curve]{xy}
\usepackage{combelow}
\usepackage{newunicodechar}
\usepackage[pagewise,displaymath,mathlines,left]{lineno}
\usepackage{tikz}

\newcommand{\pgftextcircled}[1]{
    \setbox0=\hbox{#1}%
    \dimen0\wd0%
    \divide\dimen0 by 2%
    \begin{tikzpicture}[baseline=(a.base)]%
        \useasboundingbox (-\the\dimen0,0pt) rectangle (\the\dimen0,1pt);
        \node[circle,draw,outer sep=0pt,inner sep=0.1ex] (a) {#1};
    \end{tikzpicture}
}

\DeclareMathOperator{\BDM}{\mathbf{DM}}

\DeclareMathOperator{\DM}{DM}

\newcommand\implik{{\ \Longrightarrow\ }}
\newcommand\ekviv{{\ \Longleftrightarrow\ }}

\makeindex             

\begin{document}

\title*{Residuated operators and Dedekind-MacNeille completion}
\author{Ivan~Chajda, Helmut~L\"anger and Jan Paseka}
\institute{Ivan Chajda \at 
Palack\'y University Olomouc,  
Faculty of Science,  
Department of Algebra and Geometry, 
17.\ listopadu 12, 
771 46 Olomouc,
Czech Republic, 
\email{ivan.chajda@upol.cz}
\and Helmut L\"anger \at 
TU Wien, 
Faculty of Mathematics and Geoinformation, 
Institute of Discrete Mathematics and Geometry, 
Wiedner Hauptstra\ss e 8-10, 
1040 Vienna, 
Austria, and 
Palack\'y University Olomouc,  
Faculty of Science,  
Department of Algebra and Geometry, 
17.\ listopadu 12, 
771 46 Olomouc,
Czech Republic, 
\email{helmut.laenger@tuwien.ac.at}
\and Jan Paseka \at  
Masaryk University, 
Faculty of Science, 
Department of Mathematics and Statistics, 
Kotl\'a\v rsk\'a 2, 
611 37 Brno, 
Czech Republic,
\email{paseka@math.muni.cz}%
}
%
%
\maketitle

\abstract*{The concept of operator residuation for bounded posets with 
unary operation was introduced by the first two authors. It turns out that in some cases 
when these operators are transformed into lattice terms and the poset 
${\mathbf P}$ is completed into a Dedekind-MacNeille completion $\BDM(\mathbf P)$ 
then the complete lattice $\BDM(\mathbf P)$ becomes a residuated 
lattice with respect to these transformed terms.  It is shown that 
this holds in particular for Boolean posets and for relatively 
pseudocomplemented posets. More complicated situation is with orthomodular 
and pseudo-orthomodular posets. We show which operators $M$ (multiplication) 
and $R$ (residuation) yield operator left-residuation in a 
 pseudo-orthomodular poset ${\mathbf P}$ and if  $\BDM(\mathbf P)$ is  
 an orthomodular lattice then the transformed lattice terms $\odot$ 
 and $\to$ form a left residuation in $\BDM(\mathbf P)$. However, it is 
 a problem to determine when $\BDM(\mathbf P)$ is an orthomodular lattice. 
 We get some classes of  pseudo-orthomodular posets for which 
 their Dedekind-MacNeille completion is  an orthomodular lattice and 
 we introduce the so called strongly $D$-continuous 
 pseudo-orthomodular posets. Finally we prove that, for 
 a  pseudo-orthomodular poset ${\mathbf P}$, the 
 Dedekind-MacNeille completion $\BDM(\mathbf P)$ is  an orthomodular lattice 
 if and only if  ${\mathbf P}$ is strongly $D$-continuous. }

\abstract{The concept of operator residuation for bounded posets with 
unary operation was introduced by the first two authors. It turns out that in some cases 
when these operators are transformed into lattice terms and the poset 
${\mathbf P}$ is completed into a Dedekind-MacNeille completion $\BDM(\mathbf P)$ 
then the complete lattice $\BDM(\mathbf P)$ becomes a residuated 
lattice with respect to these transformed terms. It is shown that 
this holds in particular for Boolean posets and for relatively 
pseudocomplemented posets. More complicated situation is with orthomodular 
and pseudo-orthomodular posets. We show which operators $M$ (multiplication) 
and $R$ (residuation) yield operator left-residuation in a 
 pseudo-orthomodular poset ${\mathbf P}$ and if  $\BDM(\mathbf P)$ is  
 an orthomodular lattice then the transformed lattice terms $\odot$ 
 and $\to$ form a left residuation in $\BDM(\mathbf P)$. However, it is 
 a problem to determine when $\BDM(\mathbf P)$ is an orthomodular lattice. 
 We get some classes of  pseudo-orthomodular posets for which 
 their Dedekind-MacNeille completion is  an orthomodular lattice and 
 we introduce the so called strongly $D$-continuous 
 pseudo-orthomodular posets. Finally we prove that, for 
 a  pseudo-orthomodular poset ${\mathbf P}$, the 
 Dedekind-MacNeille completion $\BDM(\mathbf P)$ is  an orthomodular lattice 
 if and only if  ${\mathbf P}$ is strongly $D$-continuous. 
}

\section{Introduction}
\label{sec:1}
Consider a bounded poset $\mathbf P=(P,\leq,{}',0,1)$ with a unary operation $'$. For $M\subseteq P$ denote by
\begin{align*}
U(M) & :=\{x\in P\mid y\leq x\text{ for all }y\in M\},
\end{align*}
the so-called {\em upper cone of} $M$, and by 
\begin{align*}
L(M) & :=\{x\in P\mid x\leq y\text{ for all }y\in M\}, 
\end{align*}
the so-called {\em lower cone of} $M$. If $M=\{a,b\}$ or $M=\{a\}$, we will write simply $U(a,b)$, $L(a,b)$ 
or $U(a)$, $L(a)$, respectively.

The following concept was introduced in \cite{CLRepo}. 

\begin{definition}\label{leftresid}{\rm 
An {\em operator left residuated poset} is an ordered seventuple 
$\mathbf P=(P,\leq,$ ${}',M,R,0,$ $1)$ where $(P,\leq,{}',0,1)$ is 
a bounded poset with a unary operation and $M$ and $R$ are 
mappings from $P^2$ to $2^P$ satisfying the following conditions for all $x,y,z\in P$:
\begin{eqnarray}
& & M(x,1)\approx M(1,x)\approx L(x),\label{equ1} \\
& & M(x,y)\subseteq L(z)\text{ if and only if }L(x)\subseteq R(y,z),\label{equ3} \\
& & R(x,0)\approx L(x').
\end{eqnarray}}
\end{definition}

It is elementary to show that 
$$
R(x,y)=P\text{\, if and only if \,}x\leq y.
$$

In what follows, we will work with posets $\mathbf P=(P,\leq,{}',0,1)$ where $'$ is 
an antitone involution  or a complemetation.  The precise definition is the following. 

\begin{definition}\label{dualcomp}{\rm A {\em poset with antitone involution} is an ordered quintuple 
$\mathbf P=(P,$ $\mbox{$\leq,$ ${}',0,1)$}$ such that $(P,\leq,0,1)$ is a bounded poset and $'$ is 
a unary operation on $P$ satisfying the following conditions for all $x,y\in P$:
\begin{enumerate}[{\rm(i)}]
\item $x\leq y$ implies $y'\leq x'$,
\item $(x')'\approx x$.
\end{enumerate}

A {\em poset with complementation} is a poset with 
antitone involution $\mathbf P=(P,$ \mbox{$\leq,{}',0,1)$}
satisfying the following LU-identities:
\begin{enumerate}
\item[{\rm(iii)}] $L(x,x')\approx\{0\}$ and $U(x,x')\approx\{1\}$.
\end{enumerate}}

A subset $S \subseteq P$ of a   poset  $\mathbf P$ with complementation such 
that $s \leq t'$ for any pair $s, t\in S, s\not= t$ is called {\em orthogonal}.  
$\mathbf P$ is said to have a {\em finite rank} if every orthogonal subset of $\mathbf P$ 
is finite.
\end{definition}

A natural and interesting question is for which posets $\mathbf P=(P,\leq,{}',0,1)$  the operators 
$M$ and $R$ can be constructed by means of the operators $L$ and $U$ similarly 
as in (left) residuated lattices the operations $\odot$ and $\to$ can be expressed as term 
operations.

For reader's convenience we recall that a lattice $\mathbf L=(L,\vee,\wedge,1)$ with the 
greatest element $1$ is {\em left residuated} if there are two binary 
operations $\odot$ and $\to$ on $L$ such that for all $x, y, z\in L$ we have 
\begin{eqnarray}
& &x\odot 1\approx x\approx 1\odot x,\label{requ1} \\
& &x\odot y\leq z\text{\, if and only if \,}x\leq y\to z.\label{requ2}
\end{eqnarray}

In our treaty we do not ask that $\odot$ has to be associative, i.e., 
it need not be a t-norm.
If $\odot$ is commutative then 
we simply say that $\mathbf L$ is {\em residuated}.

It was shown by the first two authors 
in \cite{CLRela} that this is the case for Boolean algebras, 
orthomodular lattices and, as it is familiarly known, for relatively pseudocomplemented 
lattices. 

For every poset $\mathbf P=(P,\leq)$, its Dedekind-MacNeille completion 
$\BDM(\mathbf P)$ is a complete lattice. In what follows, we say that 
an expression in operators $U$ and $L$ is $\BDM$-{\em transformed} 
if every expression $U(x,y)$ or $LU(x,y)$ is substituted by $x\vee y$ and  
every expression $L(x,y)$ is replaced by $x\wedge y$.

The aim of this paper is as follows. Having an operator  left residuated poset 
$\mathbf P=(P,\leq,$ ${}',M,R,0,$ $1)$ we ask whether the 
operators $M$ and $R$ expressed in $U$ and $L$ can be 
$\BDM$-{transformed} such that the resulting expressions will be 
 binary operations $\odot$ and $\to$ on $\BDM(\mathbf P)$ 
 satisfying (\ref{requ1}) and (\ref{requ2}) in the 
  Dedekind-MacNeille completion  $\BDM(\mathbf P)$ of $\mathbf P$.

\section{Dedekind-MacNeille completion}

In this section, we shall discuss  several important classes of 
bounded posets ${\mathbf P}$ with a unary operation which are 
operator residuated or operator left residuated and, 
moreover, the operator residuation from  ${\mathbf P}$  
can be transformed into the  residuation in $\BDM(\mathbf P)$  
by replacing $UL$-terms of ${\mathbf P}$ into lattice terms 
of $\BDM(\mathbf P)$.

We start with detailed definitions of these concepts. 

It is well-known that every poset $(P,\leq)$ can be embedded into a complete lattice ${\mathbf L}$. 
We frequently take the so-called 
Dedekind-MacNeille completion $\BDM(P,\leq)$ for this ${\mathbf L}$. 

Hence, let ${\mathbf P} = (P,\leq)$ be a poset. Put 
$\DM({\mathbf P}):=\{B\subseteq P\mid LU(B)=B\}$. 
(We simply write $LU(B)$ instead of $L(U(B))$. Analogous simplifications are used in the sequel.) 
Then for $\DM({\mathbf P})=\{L(B)\mid B\subseteq P\}$, 
$\BDM({\mathbf P}):=(\DM({\mathbf P}),\subseteq)$ is 
a complete lattice and $x\mapsto L(x)$ is an embedding from 
$\mathbf P$ to $\BDM({\mathbf P})$ 
preserving all existing joins and meets,  and an order isomorphism 
between posets $\mathbf P$ and $(\{L(x)\mid x\in P\},\subseteq)$. 
We usually identify $P$  with $\{ L(x) \mid x\in P\}$.

For subsets $B$ and $C$ of a poset $(P,\leq)$ we will write
$B\leq C$ if and only if  $b\leq c$ for all $b\in B$ and $c\in C$. We write
$b\leq C$ instead of $\{b\}\leq C$ and $B\leq c$ instead of
$B\leq\{c\}$.

It is easy to see that 
if $B, C\subseteq P$ such that $B\leq C$ 
then $\bigvee_{\BDM({\mathbf P})} B=\bigwedge_{\BDM({\mathbf P})} C$ if and only if 
$\{x\in P \mid x\leq C\} \leq \{y\in P \mid B\leq y\}$. 

By Schmidt \cite{Schmidt}
the Dedekind-MacNeille completion of a poset ${\mathbf P}$ is
(up to isomorphism) any complete lattice ${\mathbf L}$ into which ${\mathbf P}$ can be
supremum-densely and infimum-densely embedded (i.e., for every
element $x\in L$ there exist $M,Q\subseteq P$ such that
$x=\bigvee\varphi(M)=\bigwedge\varphi(Q)$,
where $\varphi\colon{}P\to L$ is the embedding).

Let $\mathbf P$ be equipped with a binary operation $*$. We introduce a new operation $\bm*$ on $\DM({\mathbf P})$ as follows:
\[
X\bm*Y:=\bigcap_{a\in X,b\in U(Y)}L(a*b)
\]
for all $X,Y\in\DM({\mathbf P})$.

Recall that a poset $(P,\leq)$ is called 
{\em relatively pseudocomplemented} if for each 
$a,b\in P$ there exists a greatest element $c$ of $P$ 
satisfying $L(a,c)\subseteq L(b)$, see e.g.\ \cite{CLP}. 
This element $c$ is called the {\em relative pseudocomplement} of $a$ 
with respect to $b$ and it is denoted by $a*b$. Every relative pseudocomplemented poset has a greatest element $1$ since $x*x=1$ for every $x\in P$.

The following is known.

\begin{proposition}{\rm\cite[Theorem 3.1]{CLP}}\label{BDMpse} Let 
$\mathbf P=(P,\leq)$ be a poset and $*$ a binary operation on $P$. 
Then the following are equivalent:
\begin{enumerate}
\item[{\rm(i)}] $\mathbf P$ has the top element $1$ and 
$(P,*,1)$ is a relatively pseudocomplemented poset;
\item[{\rm(ii)}] $(\DM(\mathbf P), \bm*, P)$ is a 
relatively pseudocomplemented lattice satisfying the LU-identity 
\[
L(x)\bm*L(y)= L(x*y).
\]
\end{enumerate}
\end{proposition}

Recall that a {poset} $\mathbf P$ is  {\em distributive} if 
it satisfies one of the following equivalent identities:
\begin{align*}
L(U(x,y),z) & \approx LU(L(x,z),L(y,z)), \\
U(L(x,y),z) & \approx UL(U(x,z),U(y,z)).
\end{align*}

A bounded poset $\mathbf P=(P,\leq,$ ${}',0,1)$ is called {\em Boolean} if 
it is a  distributive poset and $'$ is the complementation.

\begin{example}
Fig.\ 1 shows two Boolean posets which are not Boolean algebras.

\begin{figure}[htbp]\centering
\vspace*{1.4mm}
\[
\setlength{\unitlength}{7mm}
\begin{picture}(6,8)
\put(3,0){\circle*{.3}}
\put(0,2){\circle*{.3}}
\put(2,2){\circle*{.3}}
\put(4,2){\circle*{.3}}
\put(6,2){\circle*{.3}}
\put(0,4){\circle*{.3}}
\put(2,4){\circle*{.3}}
\put(4,4){\circle*{.3}}
\put(6,4){\circle*{.3}}
\put(0,6){\circle*{.3}}
\put(2,6){\circle*{.3}}
\put(4,6){\circle*{.3}}
\put(6,6){\circle*{.3}}
\put(3,8){\circle*{.3}}
\put(3,0){\line(-3,2)3}
\put(3,0){\line(-1,2)1}
\put(3,0){\line(1,2)1}
\put(3,0){\line(3,2)3}
\put(3,8){\line(-3,-2)3}
\put(3,8){\line(-1,-2)1}
\put(3,8){\line(1,-2)1}
\put(3,8){\line(3,-2)3}
\put(0,2){\line(0,1)4}
\put(6,2){\line(0,1)4}
\put(0,4){\line(1,1)2}
\put(0,2){\line(1,1)4}
\put(2,2){\line(1,1)4}
\put(4,2){\line(1,1)2}
\put(2,2){\line(-1,1)2}
\put(4,2){\line(-1,1)4}
\put(6,2){\line(-1,1)4}
\put(6,4){\line(-1,1)2}
\put(2.85,-.75){$0$}
\put(-.6,1.9){$a$}
\put(1.2,1.9){$b$}
\put(4.45,1.9){$c$}
\put(6.4,1.9){$d$}
\put(-.6,3.9){$e$}
\put(1.2,3.9){$f$}
\put(4.45,3.9){$f'$}
\put(6.4,3.9){$e'$}
\put(-.7,5.9){$d'$}
\put(1.2,5.9){$c'$}
\put(4.45,5.9){$b'$}
\put(6.4,5.9){$a'$}
\put(2.85,8.4){$1$}
\put(2.6,-1.5){$(a)$}
\end{picture}
\quad\quad\quad\quad
\setlength{\unitlength}{7mm}
\begin{picture}(6,8)
\put(3,0){\circle*{.3}}
\put(0,2){\circle*{.3}}
\put(2,2){\circle*{.3}}
\put(4,2){\circle*{.3}}
\put(6,2){\circle*{.3}}
\put(0,4){\circle*{.3}}
\put(6,4){\circle*{.3}}
\put(0,6){\circle*{.3}}
\put(2,6){\circle*{.3}}
\put(4,6){\circle*{.3}}
\put(6,6){\circle*{.3}}
\put(3,8){\circle*{.3}}
\put(3,0){\line(-3,2)3}
\put(3,0){\line(-1,2)1}
\put(3,0){\line(1,2)1}
\put(3,0){\line(3,2)3}
\put(3,8){\line(-3,-2)3}
\put(3,8){\line(-1,-2)1}
\put(3,8){\line(1,-2)1}
\put(3,8){\line(3,-2)3}
\put(0,2){\line(0,1)4}
\put(6,2){\line(0,1)4}
\put(0,4){\line(1,1)2}
\put(0,2){\line(1,1)4}
\put(2,2){\line(1,1)4}
\put(4,2){\line(1,1)2}
\put(2,2){\line(-1,1)2}
\put(4,2){\line(-1,1)4}
\put(6,2){\line(-1,1)4}
\put(6,4){\line(-1,1)2}
\put(2.85,-.75){$0$}
\put(-.6,1.9){$a$}
\put(1.2,1.9){$b$}
\put(4.45,1.9){$c$}
\put(6.4,1.9){$d$}
\put(-.6,3.9){$e$}
\put(6.4,3.9){$e'$}
\put(-.7,5.9){$d'$}
\put(1.2,5.9){$c'$}
\put(4.45,5.9){$b'$}
\put(6.4,5.9){$a'$}
\put(2.85,8.4){$1$}
\put(2.6,-1.5){$(b)$}
\put(-2,-2){{\rm Fig.\ 1}}
\end{picture}
\]
\vspace*{9mm}
\end{figure}

\end{example}

The following result was proved by Niederle \cite{Niederle}.

\begin{proposition}{\rm \cite[Theorem 16]{Niederle}}\label{niederle}  
For every Boolean poset $\mathbf P=(P,\leq, {}',0,1)$  its 
Dedekind-MacNeille completion $\BDM(\mathbf P)$ is a complete Boolean 
algebra.
\end{proposition}

Unfortunately, for other interesting classes of posets we do not have 
such a nice result. A {\em poset}  with complementation  $\mathbf P=(P,\leq,{}',0,1)$  
is called {\em orthomodular} if for all $x, y\in P$ with 
$x\leq y'$ there exists $x\vee y$ and then $\mathbf P$ satisfies 
 one of the following equivalent identities:
\begin{align*}
((x\wedge y)\vee y')\wedge y & \approx x\wedge y, \\
((x\vee y)\wedge y')\vee y & \approx x\vee y
\end{align*}
where $x\wedge y$ stands for $(x'\vee y')'$ 
(De Morgan laws). 

It is known that for an  orthomodular poset  $\mathbf P=(P,\leq,{}',0,1)$, 
its Dedekind-MacNeille completion $\BDM(\mathbf P)$  need not be an 
 orthomodular  lattice. 
 
Recall that a lattice  with complementation 
$(L,\wedge, \vee, {}',0,1)$ is orthomodular  if and only if 
 it satisfies the following 
identity \cite[Theorem II.5.1]{beran}:
\begin{align*}
x\vee y & \approx  ((x\vee y)\wedge y')\vee y.
\end{align*}
which in turn is equivalent to the following condition (\cite[Chapter 1, 2. Theorem]{kalmb83}): 
\begin{center}
if $x, y\in L$, $x\leq y$ and $x'\wedge y=0$ then $x=y$.
\end{center}

The poset $\mathbf P$ with complementation is called an {\em orthocomplete poset} 
if $\bigvee S$ exists in  $\mathbf P$ for every orthogonal subset $S \subseteq P$.
 
 The poset $\mathbf P$ with complementation is called 
a {\em pseudo-orthomodular poset}\/ if it satisfies one of the following equivalent 
conditions:
\begin{align*}
L(U(L(x,y),y'),y) & \approx L(x,y), \\
U(L(U(x,y),y'),y) & \approx U(x,y).
\end{align*}

It is worth noticing that if the previous expressions are 
$\BDM$-{transformed} we obtain the orthomodular  law which holds 
in orthomodular lattices. Unfortunately, if 
$\mathbf P=(P,\leq,$ ${}',0,1)$ is a  pseudo-orthomodular poset then its 
Dedekind-MacNeille completion $\BDM(\mathbf P)$  need not be an 
 orthomodular  lattice. 
 
 Of course, every Boolean poset is pseudo-orthomodular and every 
 orthomodular lattice is a pseudo-orthomodular poset.
 
 We can state and prove the following result.
 
 \begin{theorem}\label{boolpos} Let $\mathbf P=(P,\leq, {}',0,1)$ be a 
  Boolean poset. Take $M(x,y)=L(x,y)$ and $R(x,y)=L(U(x',y))$. 
  Then
  \begin{enumerate}[{\rm(i)}]
  \item $\mathbf P$ is operator residuated with respect to $M$ and $R$;
  \item $\BDM(\mathbf P)$ is a complete Boolean algebra which is a residuated 
  lattice with respect to the operations $\odot$ and $\to$ reached by the 
  $\BDM$-transformation from $M$ and $R$, respectively, i.e., 
  $x\odot y=x\wedge y$ and $x\to y=x'\vee y$. 
  \end{enumerate}
 \end{theorem}
 \begin{proof} \smartqed
 (i) is proved in \cite{CLRepo}, the first part of (ii) is shown by Proposition 
 \ref{niederle}, the $\BDM$-transformation is evident and the fact that 
 $\BDM(\mathbf P)$ is a residuated 
  lattice with respect to the operations $\odot$ and $\to$ is well-known. 
  \qed
 \end{proof}
 
 Similar results can be stated for relatively pseudocomplemented posets.
 
 Recall that a lattice $\mathbf  L = (L, \vee, \wedge)$ is 
 relatively pseudocomplemented if for each $a,b \in L$ 
 there exists the greatest element of the set $\{x \in L\mid  a \wedge  x \leq b\}$, 
 the so-called {\em relative pseudocomplement of $a$ with respect to $b$}; 
 it is denoted by $a*b$. Evidently,
$$       a \wedge b  \leq c\   \text{ if and only if }\   a \leq b*c. $$
 
  \begin{theorem}\label{relpspos} Let $(P,\leq, *,0,1)$ be a 
  relatively pseudocomplemented  poset. 
  Take $x'=x*0$, $M(x,y)=L(x,y)$ and $R(x,y)=L(x*y)$. 
  Then
  \begin{enumerate}[{\rm(i)}]
  \item $\mathbf P=(P,\leq,\, {}',M,R,0,1)$ is operator residuated;
  \item $\BDM(\mathbf P)$ is a complete relatively pseudocomplemented 
  lattice which is a residuated 
  lattice with respect to the operations $\odot$ and $\to$ reached by the 
  $\BDM$-transformation from $M$ and $R$, respectively, i.e., 
  $x\odot y=x\wedge y$ and $x\to y=x\bm{*} y$. 
  \end{enumerate}
 \end{theorem}

We need not get a proof because every of these assertions is 
familiarly known. Namely, 
\begin{eqnarray*}
M(x,y)\subseteq L(z) &\Longleftrightarrow& L(x,y) \subseteq L(z)  \Longleftrightarrow 
L(x)\subseteq L(y*z) \Longleftrightarrow L(x)\subseteq  R(y,z). 
\end{eqnarray*}
It was shown by the authors in \cite{CLP} that the pseudocomplementation 
$*$ in  $\BDM(\mathbf P)$  for elements from $P$ is the same as in 
$\mathbf P$. 

\section{Completion of pseudo-orthomodular posets}

As mentioned above, the lattice $\BDM(\mathbf P)$ for a pseudo-orthomodular poset 
$\mathbf P=(P,\leq, {}',0,1)$ need not be an orthomodular lattice. It was shown 
in \cite{CLRepo} that  for $M(x,y)=L(U(x,y'),y)$ and $R(x,y)=L(U(L(x,y),x'))$,  
$\mathbf P$ becomes an operator left residuated poset. Unfortunately, 
making $\BDM$-transformation of $M$ and $R$, the 
Dedekind-MacNeille completion $\BDM(\mathbf P)$  need not be a left 
residuated lattice with respect to $x\odot y=(x\vee y')\wedge y$ and 
$x\to y=(x\wedge y)\vee x'$ despite the fact that every orthomodular lattice 
is left residuated with respect to these operations. 

The aim of this section is to show some cases of posets $\mathbf P$ 
for which $\BDM(\mathbf P)$ is 
an orthomodular lattice and when $\BDM$-transformation of $M$ and $R$ 
yields operations  $\odot$ and $\to $ such that $\BDM(\mathbf P)$ 
is  a left residuated lattice.

The horizontal sum of a family of bounded posets is obtained
from their disjoint union by identifying the top elements and the bottom elements, respectively. Note that 
a horizontal sum of a family of bounded posets with 
antitone involution (complementation) is a bounded poset with antitone involution 
(complementation), respectively.

\begin{proposition}\label{dmhosum}Let $\mathbf P=(P,\leq,0,1)$ 
be  a bounded poset such that 
$\mathbf P$ is a horizontal sum of bounded 
posets $\mathbf P_{\alpha}=(P_{\alpha},\leq_{\alpha},0,1)$, 
$\alpha\in\Lambda$. 
Then $\BDM({\mathbf P})$ is order-isomorphic to  a horizontal sum $\mathbf Q$ of 
complete lattices  
$\BDM({\mathbf P_{\alpha}})$, $\alpha\in\Lambda$.  
\end{proposition}
\begin{proof} \smartqed 
Clearly, a horizontal sum of complete lattices is a complete lattice. Moreover, 
$\mathbf P=(P,\leq,0,1)$ is both join-dense and meet-dense in $\mathbf Q$ and we have 
an order embedding from $\mathbf P$ into $\mathbf Q$. It follows that 
$\BDM({\mathbf P})$ is order-isomorphic to   $\mathbf Q$. 
\qed
\end{proof}

Using this, we can prove the following result. 

\begin{proposition}\label{xxdmhosum}Let $\mathbf P=(P,\leq,{}',0,1)$ be  a bounded poset such that 
$\mathbf P$ is a horizontal sum of pseudo-orthomodular posets $\mathbf P_{\alpha}=(P_{\alpha},\leq_{\alpha},{}^{{'}_{\alpha}},0,1)$, 
$\alpha\in\Lambda$. 
Then $\mathbf P$ is a pseudo-orthomodular poset. 
\end{proposition}
\begin{proof} \smartqed If $x\in \{0,1\}$  or $y\in \{0,1\}$ then clearly 
$L(U(L(x,y),y'),y) \approx L(x,y)$. Assume that $x, y\in P\setminus\{0,1\}$.
Suppose first that $x\in P_{\alpha}\setminus\{0,1\}$ and $y\in P_{\beta}\setminus\{0,1\}$, 
$\alpha, \beta\in\Lambda$, $\alpha\not= \beta$. It follows that $ L(x,y)=\{0\}$. Hence 
$U(L(x,y),y')=U(y')$ and $L(U(y'),y)=\{0\}$, i.e., we have again $L(U(L(x,y),y'),y) \approx L(x,y)$.
To the end, assume that  $x, y\in P_{\alpha}\setminus\{0,1\}$. We have 
$L(x,y)=L_{P_{\alpha}}(x,y)$, $U(L(x,y),y')=U_{P_{\alpha}}(L_{P_{\alpha}}(x,y),y')$ 
and $L(U(L(x,y),y'),y)=L_{P_{\alpha}}(U_{P_{\alpha}}(L_{P_{\alpha}}(x,y),y'),y)$.    
This yields  $L(U(L(x,y),y'),y) \approx L(x,y)$ since  $\mathbf P_{\alpha}$ 
is a pseudo-orthomodular poset. 
\qed
\end{proof}

\begin{theorem}\label{ppdmhosum}Let $\mathbf P=(P,\leq,{}',0,1)$ be  a bounded poset such that 
$\mathbf P$ is a horizontal sum of pseudo-orthomodular posets $\mathbf P_{\alpha}=(P_{\alpha},\leq_{\alpha},{}^{{'}_{\alpha}},0,1)$, 
$\alpha\in\Lambda$, and any $\BDM({\mathbf P}_{\alpha})$ is a complete orthomodular lattice. 
Then $\BDM({\mathbf P})$ is a complete orthomodular lattice. 
\end{theorem}
\begin{proof} \smartqed From Proposition \ref{dmhosum} we know that  $\BDM({\mathbf P})$  is order-isomorphic to 
  a horizontal sum $\mathbf Q$ of complete lattices 
$\BDM({\mathbf P_{\alpha}})$. It is evident that the isomorphism preserves 
the antitone involution  as well. Since 
any $\BDM({\mathbf P}_{\alpha})$ is a complete orthomodular lattice we have 
that $\BDM({\mathbf P})$ is orthomodular. \qed
\end{proof}

We obtain the following corollary of Theorem \ref{ppdmhosum} and Proposition \ref{niederle}. 

\begin{corollary}\label{cccdmhosum}Let $\mathbf P=(P,\leq,{}',0,1)$ be  a bounded poset such that 
$\mathbf P$ is a horizontal sum of Boolean posets 
$\mathbf P_{\alpha}=(P_{\alpha},\leq_{\alpha},{}^{{'}_{\alpha}},0,1)$,  $\alpha\in\Lambda$,  
and $M(x,y) = L(U(x,y'),y)$  and $R(x,y) = LU(L(x,y),x')$. 
Then $\BDM({\mathbf P})$ is a complete orthomodular lattice.  Moreover, 
 $\BDM(\mathbf P)$ is  a left residuated lattice with respect to $\odot$ and $\to$ reached by the 
 $\BDM$-transformation from $M$ and $R$, respectively.  
\end{corollary}

The proof of the last assertion in Corollary \ref{cccdmhosum} follows 
from the fact that every orthomodular lattice is  a left residuated lattice with respect to 
$x\odot y=(x\vee y')\wedge y$ and 
$x\to y=(x\wedge y)\vee x'$, see \cite{CLRela}  for details.

Hence, horizontal sums of non-trivial Boolean posets  form a class of 
pseudo-orthomodular posets which can be extended to an orthomodular lattice 
and the residuation of the latter can be reached by the 
$\BDM$-transformation. 

\begin{example}\rm 
Consider the horizontal sum $\mathbf P$ of the Boolean poset 
 $\mathbf P_1$ where $P_1=\{0, a, b, c, d, e, e', d', c', b', a', 1\}$ 
  and an four-element Boolean algebra  $\mathbf P_2$ 
  where $P_2=\{0,f,$ $f',1\}$  and whose Hasse diagram is depicted in {\rm Fig.\ 2}:

\vspace*{3mm}

\[
\setlength{\unitlength}{6.6mm}
\begin{picture}(18,8)
\put(9,0){\circle*{.3}}
\put(6,2){\circle*{.3}}
\put(8,2){\circle*{.3}}
\put(10,2){\circle*{.3}}
\put(12,2){\circle*{.3}}
\put(6,4){\circle*{.3}}
\put(12,4){\circle*{.3}}
\put(6,6){\circle*{.3}}
\put(8,6){\circle*{.3}}
\put(10,6){\circle*{.3}}
\put(12,6){\circle*{.3}}
\put(9,8){\circle*{.3}}
\put(1,4){\circle*{.3}}
\put(17,4){\circle*{.3}}
\put(9,0){\line(-3,2)3}
\put(9,0){\line(-1,2)1}
\put(9,0){\line(1,2)1}
\put(9,0){\line(3,2)3}
\put(9,8){\line(-3,-2)3}
\put(9,8){\line(-1,-2)1}
\put(9,8){\line(1,-2)1}
\put(9,8){\line(3,-2)3}
\put(6,2){\line(0,1)4}
\put(12,2){\line(0,1)4}
\put(6,4){\line(1,1)2}
\put(6,2){\line(1,1)4}
\put(8,2){\line(1,1)4}
\put(10,2){\line(1,1)2}
\put(8,2){\line(-1,1)2}
\put(10,2){\line(-1,1)4}
\put(12,2){\line(-1,1)4}
\put(12,4){\line(-1,1)2}
\put(1,4){\line(2,-1)8}
\put(1,4){\line(2,1)8}
\put(17,4){\line(-2,-1)8}
\put(17,4){\line(-2,1)8}
\put(8.85,-.75){$0$}
\put(5.4,1.9){$a$}
\put(7.2,1.9){$b$}
\put(10.45,1.9){$c$}
\put(12.4,1.9){$d$}
\put(5.4,3.9){$e$}
\put(.4,3.9){$f$}
\put(12.4,3.9){$e'$}
\put(17.4,3.9){$f'$}
\put(5.3,5.9){$d'$}
\put(7.2,5.9){$c'$}
\put(10.45,5.9){$b'$}
\put(12.4,5.9){$a'$}
\put(8.85,8.4){$1$}
\put(8.2,-1.5){{\rm Fig.\ 2}}
\end{picture}
\]

\vspace*{8mm}

According to Proposition~\ref{xxdmhosum} and 
Corollary \ref{cccdmhosum}, $\mathbf P$  
 is a pseudo-orthomodular poset 
and $\BDM({\mathbf P})$  is a nonmodular orthomodular lattice.
\end{example}

We can solve our problem also from the opposite direction. Namely, we can assume that 
$\BDM({\mathbf P})$ is really an orthomodular lattice and ask what is  $\mathbf P$. 
The answer is as follows.

\begin{theorem}\label{xTheorem4.5}  Let $\mathbf P=(P,\leq,{}',0,1)$ 
be a complemented poset such that $\BDM(\mathbf P)$ is an orthomodular 
lattice. Then $\mathbf P$ is pseudo-orthomodular.
\end{theorem}
\begin{proof}\smartqed Let $\BDM(\mathbf P)$ be an orthomodular 
lattice and let $x, y\in P$. We compute:
$$
L(U(x,y))=x\vee_{\BDM(\mathbf P)}y=%
((x\vee_{\BDM(\mathbf P)} y)\wedge_{\BDM(\mathbf P)} y')%
\vee_{\BDM(\mathbf P)} y%
=LU(L(U(x,y),y'),y).
$$
It follows that $U(L(U(x,y),y'),y) = U(x,y)$, i.e., $\mathbf P$ is 
 pseudo-orthomodular.\qed
\end{proof}

Let us note that the result of Theorem \ref{xTheorem4.5} justifies the concept 
of a pseudo-orthomodular poset. With respect to the completion into 
an orthomodular lattice it is more appropriate than the concept 
of  an orthomodular poset. It will be emphasized also by Corollary \ref{dusl2} 
and Theorem \ref{apfinch} below.

In what follows we will show that  for finite orthomodular posets 
 $\mathbf P$ such that  $\mathbf P$   is not a lattice
  their Dedekind-MacNeille completions $\BDM(\mathbf P)$   are not orthomodular. 
 
 We will need the following definitions and theorem  from Kalmbach
\cite{kalmb83} reformulated as in  Svozil and Tkadlec \cite{svozil-tkadlec}.

\begin{definition}\label{D:diagram}{\rm 
A\/ {\em diagram} is a pair $(V,E)$, where $V\ne\emptyset$ is a set of\/
{\em atoms} (drawn as points) and
$E\subseteq {\rm exp}\,V\>\backslash\,\{\emptyset\}$ is a set of\/
{\em blocks} (drawn as line segments connecting corresponding points).
A\/ {\em loop} of order $n\ge 2$ ($n$ being a natural number) in a
diagram $(V,E)$ is a sequence $(e_1,\dots ,e_n)\in E^n$ of mutually
different blocks such that there are mutually distinct
atoms $\nu_1,\dots,\nu_n$ with $\nu_i\in e_i\cap e_{i+1}\
(i=1,\dots,n,\ e_{n+1}=e_1)$.}
\end{definition}

In particular, we precise it as follows (see e.g. \cite{kalmb83}).

\begin{definition}\label{D:greechie-diagram}{\rm 
A\/ {\em Greechie diagram} is a diagram satisfying the following conditions:
\begin{enumerate}
\item[(1)] Every atom belongs to at least one block.
\item[(2)] If there are at least two atoms then every block is at least
2-element.
\item[(3)] Every block which intersects with another block is at least
3-element.
\item[(4)] Every pair of different blocks intersects in at most one atom.
\item[(5)] There is no loop of order 3.
\end{enumerate}}
\end{definition}

Recall that a {\em block} in an orthomodular poset  is a maximal Boolean 
subalgebra of it. An element $a$ of a poset $\mathbf P$ with least element 
$0$ is an {\em atom} if $0 < a$ and there is no $x\in P$ such that 
$0 < x < a$. A poset $\mathbf P$ with a least element $0$ is 

\begin{enumerate}[(i)]
\item {\em atomic} if every element $b > 0$ has an atom $a$ below it, 
\item {\em atomistic}  if every element is a join a set of atoms of $\mathbf P$. 
\end{enumerate}

\begin{theorem}\label{th:loop-lemma}{\rm \cite[Loop Lemma]{kalmb83}}
For every Greechie diagram with only finite blocks there is exactly
one (up to an isomorphism) orthomodular poset such that there are
one-to-one correspondences between atoms and atoms and between
blocks and blocks which preserve incidence relations.
The poset is a lattice if and only if the Greechie diagram has
no loops of order~4.
\end{theorem}

We use the notion {\em Greechie logic} for an orthomodular
poset that can be represented by a Greechie diagram with only finite edges. 
Recall that  every  element 
of a Greechie logic  is a supremum of a finite orthogonal set of atoms and 
suprema (infima) of elements from a block of the Greechie logic coincide with their 
suprema (infima) in the whole Greechie logic, respectively.

Using this, we can construct the promised example. 

\begin{example}\rm Let  $\mathbf P=(P,\leq,{}',0,1)$ be 
the finite Greechie logic given by the 
Greechie diagram in Fig.~3 (see also \cite[Exercise 3, page 259]{kalmb83}). 

The Greechie logic $\mathbf P$ has 
4 blocks $\mathbf B_0$, $\mathbf B_1$, $\mathbf B_2$ and $\mathbf B_3$. 
The maximal respective orthogonal sets of atoms of $\mathbf P$ are 
$\{x, y,z\}$, $\{z, t, s\}$, $\{s, u, v\}$ and $\{v, w, x\}$. 
Denote the set of all atoms of $\mathbf P$ by $A$.

\begin{figure}[htbp]\centering
  \setlength{\unitlength}{2pt}
  \begin{picture}(330,100)(0,0)
    \put(55,30){
      \begin{picture}(50,80)(0,0)
      \put(30,60){\line(1,-1){30}}
        \put(0,30){\line(1,-1){30}}
        \put(30,0){\line(1,1){30}}
        \put(0,30){\line(1,1){30}}
         \put(30,60){\circle{10}}
        \put(15,45){\circle{10}}
        \put(45,45){\circle{10}}
        \put(0,30){\circle{10}}
        \put(15,15){\circle{10}}
        \put(30,0){\circle{10}}
        \put(45,15){\circle{10}}
        \put(60,30){\circle{10}}
         \put(20,60){\makebox(0,0)[r]{$s$}}
         \put(5,45){\makebox(0,0)[r]{$u$}}
        \put(-10,30){\makebox(0,0)[r]{$v$}}
        \put(5,15){\makebox(0,0)[r]{$w$}}
        \put(30,-10){\makebox(0,0)[t]{$x$}}
        \put(55,15){\makebox(0,0)[l]{$y$}}
        \put(70,30){\makebox(0,0)[l]{$z$}}
         \put(55,45){\makebox(0,0)[l]{$t$}}
         \put(-29,-25){{\rm Fig.\ 3} Greechie diagram of a finite orthomodular 
         poset $\mathbf P$  such that} 
         \put(-29,-31){\phantom{\rm Fig.\ 3} {\rm its Dedekind-MacNeille completion 
         $\BDM(\mathbf P)$   is not orthomodular. }}
      \end{picture}
    } 

  \end{picture}
\end{figure}

We have that $y\not\in L(s', x')\cap A=\{v,z\}$. It follows that 
$U(L(s', x'),x)=U(L(s', x')\cap A,x)=\{1\}$. Hence 
$L(U(L(s', x'),x), x')=L(\{1\}, x')=L(x')$ and $y\in L(x')$. 
We conclude that $\mathbf P$   is not pseudo-orthomodular, i.e., by 
Theorem \ref{xTheorem4.5}  $\BDM(\mathbf P)$   is not orthomodular. 
\end{example}

Motivated by the above example we will prove the following.

\begin{theorem}\label{orthocomplete} Let 
$\mathbf P=(P,\leq,{}',0,1)$ be an orthocomplete atomic orthomodular poset. 
The following conditions are equivalent:
\begin{enumerate}[{\rm(i)}]
\item $\mathbf P$ is pseudo-orthomodular.
\item $\mathbf P$ is a complete orthomodular lattice.
\item $\BDM(\mathbf P)$ is orthomodular. 
\end{enumerate}
\end{theorem}
\begin{proof}\smartqed  (ii)$\implik$ (iii) is evident and (iii)$\implik$ (i) follows 
 by Theorem \ref{xTheorem4.5}.  
 
 (i)$\implik$ (ii): Let  $\mathbf P$ be a pseudo-orthomodular poset and 
 denote the set of all atoms of $\mathbf P$ by $A$. Since 
 $\mathbf P$ is an  orthocomplete atomic orthomodular poset 
 it is atomistic (namely, any element $x$ of  $\mathbf P$ is 
 a join of a maximal orthogonal set of atoms lying under $x$). Let us show that 
 $\mathbf P$ is a lattice. Assume that  $v, z\in P, v, z\not\in \{0,1\}$ (the case when 
 $v\in \{0,1\}$ or $z \in \{0,1\}$ is trivial) and let us prove that 
 $v\vee z$ exists.
 
 Suppose first that there is a maximal orthogonal set of atoms $A_1\subseteq A$ 
 such that $v=\bigvee A_v$, $z=\bigvee A_z$ and $A_v\cup A_z\subseteq A_1$. 
 We show that $v\vee z=\bigvee (A_v\cup A_z)$. Since 
 $\mathbf P$ is  orthocomplete $\bigvee (A_v\cup A_z)$ exists and 
  $v, z\leq \bigvee (A_v\cup A_z)$. Let $c\in P$,  $v, z\leq c$. Then 
  $A_v\leq c$ and $A_z\leq c$. We conclude that $A_v\cup A_z \leq c$ and 
  again by orthocompleteness of  $\mathbf P$ we have 
  $\bigvee (A_v\cup A_z)\leq c$. 
  
  Now assume that there is no maximal orthogonal set of atoms $A_1\subseteq A$ 
 such that $v=\bigvee A_v$, $z=\bigvee A_z$ and $A_v\cup A_z\subseteq A_1$. 
 
  If $v\vee z$ exists then we are finished. Assume that  $v\vee z$  does not 
 exist. From the fact that $\mathbf P$ is an orthomodular poset we have 
 $z\not\leq v'$ (equivalently,  $v\not\leq z'$). 
 
 Since $v\vee z$  does not  exist $v'\wedge z'$  does not  exist as well. 
 Hence there are two different orthogonal sets of atoms 
 $A_{\alpha}$ and   $A_{\beta}$ such that  $A_{\alpha}$ 
 and  $A_{\beta}$ are maximal 
 elements from $\{C\subseteq A\mid C\leq \{v', z'\}, 
 C \text{ orthogonal}\}$, 
  $A_{\alpha}\not\subseteq A_{\beta}$ and  $A_{\beta}\not\subseteq A_{\alpha}$.  
 Put $s=\bigvee A_{\alpha}$ and $x=\bigvee A_{\beta}$. Then 
 $s\not\leq x$,  $x\not\leq s$,
 $v\leq s'$, $v\leq x'$, $z\leq s'$ and   $z\leq x'$. Moreover, 
 $x\not\leq s'$ (equivalently,  $s\not\leq x'$). 
 
  Assume first that $s\vee v=1$. Then $s=v'$. We also 
  have $s\vee z \vee (s'\wedge z')=1$. It follows that 
  $z \vee (s'\wedge z')=v$, i.e., $z\leq v$, a contradiction with $z\not\leq v$. 
  Hence $s\vee v\not= 1$, i.e., $0<{u}=s'\wedge v'<1$. 
  Clearly, ${u}\not\leq  z'$. Otherwise we would 
  have $v'={u}\vee s\leq z'$, i.e., $z\leq v$, a contradiction. 
  By the same arguments we obtain that ${u}\not\leq  x'$. 
  Similarly by symmetry $0<{w}=x'\wedge v'<1$, 
  ${w}\not\leq  z'$ and ${w}\not\leq  s'$, 
   $0<{t}=s'\wedge z'<1$, 
  ${t}\not\leq  v'$ and ${t}\not\leq  x'$ and 
   $0<{y}=x'\wedge z'<1$, 
  ${y}\not\leq  v'$ and ${y}\not\leq  s'$. Hence we obtain the same picture as 
  in {\rm Fig.\ 3} (although the elements need not be atoms).

  Let $c\in L({u},{t})$ be an atom. Then 
  $c\leq {u}=s'\wedge v'\leq s'$,  $c\leq v'$, $c\leq z'$ and $c\not\leq s$. Hence 
   $A_{\alpha}\cup \{c\}\in \{C\subseteq A\mid C\leq \{v', z'\}, 
 C \text{ orthogonal}\}$, a contradiction with the maximality of  $A_{\alpha}$. 
 We have that ${u}\wedge {t}=0$.  Similarly, 
 ${w}\wedge {y}=0$.
  
   We assert that $\mathbf P$ is not pseudo-orthomodular. The reason is: 
  $v, z\in L(x', s')$ and 
  ${u}\not\in L(x', s')$. Let $q$ be any upper bound of the set 
  $\{v, z, s\}$. It follows that $q\geq v\vee s={u}'$ 
  and $q\geq z\vee s={t}'$. Hence 
  $q'\in  L({u},{t})=\{0\}$, i.e. $q=1$. 
  Then  it is easy to see that 
   $U(L(x', s'),  s)=\{1\}$. 
   This shows that   ${u}\in L(U(L(x', s'),  s), s')=L(\{1\}, s')=L(s')$, 
   contradicting our assumptions.  
   
   Therefore every two elements of $\mathbf P$ have a join and 
   $\mathbf P$ is a complete orthomodular lattice. \qed
   \end{proof}
 
 As our final result on orthomodular posets  
 we show that even for  a finite orthomodular poset 
  $\mathbf P$ its Dedekind-MacNeille completion  $\BDM(\mathbf P)$   
  is not orthomodular. This disqualifies these posets for 
  operator left residuation. 
 
 \begin{corollary} \label{dusl2}
 Let $\mathbf P=(P,\leq,{}',0,1)$ be a finite orthomodular poset 
 which is not a lattice. Then its Dedekind-MacNeille completion 
 $\BDM(\mathbf P)$   is not orthomodular.
 \end{corollary}
 \begin{proof} \smartqed Assume that $\BDM(\mathbf P)$   is  orthomodular. 
 From Theorem  \ref{xTheorem4.5} we have that $\mathbf P$ is 
  pseudo-orthomodular. From Theorem  \ref{orthocomplete} we obtain that 
   $\mathbf P$ is a lattice, a contradiction. \qed
 \end{proof}

 \begin{corollary} Any non-lattice Greechie logic does not possess 
  an orthomodular Dedekind-MacNeille completion.
  \end{corollary}
  
  \begin{proposition}\label{atomicpseudo}
   Let  $\mathbf P=(P,\leq,{}',0,1)$ be an  atomic  pseudo-orthomodular poset. 
   Then any element of\/ $\mathbf P$ is a join of an orthogonal set of atoms lying under it and 
   $\mathbf P$  is an atomistic poset. 
  \end{proposition}
\begin{proof} \smartqed Assume that $x\in P$ and let 
$A_x$ be a maximal orthogonal set of atoms under $x$. 
Clearly, $x\in U(A_x)$. Let $y\in U(A_x)$. We have to show that 
$x\leq y$. Evidently, $A_x\subseteq L(x, y)$. We conclude that 
$U(A_x, x')=\{1\}$. Namely, let $q\in U(A_x, x'),$ $q\not=1$. Then there is 
an atom $a\in P$ such that $a'\in U(A_x, x'), a'\geq q$. 
Consequently, $a\leq x$ and $a\leq b'$ for all $b\in A_x$, a contradiction 
with the maximality of $A_x$.

We conclude that $U(L(x, y), x')=\{1\}$, hence 
$ L(x, y)=L(U(L(x, y), x'),x)=L(\{1\}, x)=L(x)$, i.e., $x\leq y$. 
 \qed
 \end{proof}
 
 \begin{remark}\label{finrem}
 Recall that Finch \cite[Proposition  (3.2).]{Finch} has 
 shown, for a complemented poset $\mathbf P$,  that  its Dedekind-MacNeille completion  $\BDM(\mathbf P)$ 
 is orthomodular  if and only  if  for any non-empty  subset  $X$  of $\mathbf P$ 
 and any maximal orthogonal  subset $S$ of $LU(X)$ one has 
 $LU(S) =  LU(X)$.
 \end{remark}

 In Corollary \ref{dusl2} we proved that no finite non-lattice 
 orthomodular poset has an orthomodular Dedekind-MacNeille completion. 
 This is the reason why we have to modify  
 the definition of  orthomodularity in posets to obtain a more favorable result. It turns out that our concept of a pseudo-orthomodular poset can serve for this reason. Hence, we prove the following.
 
 \begin{theorem}\label{apfinch}
  Let  $\mathbf P=(P,\leq,{}',0,1)$ be an atomic pseudo-orthomodular poset with finite rank.  
  Then $\BDM(\mathbf P)$   is  orthomodular. 
 \end{theorem}
 \begin{proof} \smartqed By Remark \ref{finrem} it is enough to check that for 
  any non-empty  subset  $X$  of $\mathbf P$ 
 and any maximal orthogonal  subset $S$ of $LU(X)$ one has 
 $LU(S) =  LU(X)$.
 
 Assume that  $X\subseteq P$, $X\not=\emptyset$ and 
  $S\subseteq LU(X)$, $S$ maximal orthogonal. Since $\mathbf P$  
  has finite rank,  $S$ is finite; let $S=\{s_1, \dots, s_k\}$. 
  Put $d_S=\bigvee_{\BDM(\mathbf P)} S$. Then $d_S \leq LU(X)$. If $d_S=LU(X)$ we are done. 
  Suppose that $d_S<LU(X)$.  From Proposition \ref{atomicpseudo} we know that 
  $\mathbf P$ is atomistic. Since any element of $\BDM(\mathbf P)$ is 
  a join of elements of  $\mathbf P$ also $\BDM(\mathbf P)$ is atomistic with the same set of atoms. We conclude that there is 
  an atom $a\in P$ such that $a\not\leq d_S$ and $a\leq LU(X)$.
  
  We put 
  $$
  \begin{array}{r c l}
  l_S&=&\max\{j\in \{2, \dots, k\}\mid a\in P \text{ is an atom}, a\not\leq d_S, a\leq LU(X), \\
  & &\phantom{\max\{j\in \{2, \dots, k\}\mid }\ a\leq s_1', \dots, a\leq s_{j-1}', a\not\leq s_j'\}.
  \end{array}
  $$
  Note that $l_S$ is correctly defined since by maximality of $S$ there is no atom $a$ such that $a\leq LU(X)$ and 
  $a\leq s_1', \dots, a\leq s_{k}'$. Let $a\in P$ be an atom of   $\mathbf P$ such that  $a\not\leq d_S$, $a\leq LU(X)$, 
  $ a\leq s_1', \dots, a\leq s_{l_S-1}',$ $a\not\leq s_{l_S}'$.

  We have $LU(a,s_{l_S})=LU(L(U(a,s_{l_S}),s_{l_S}'),s_{l_S})\leq L(s_1', \dots,  s_{l_S-1}')$ since 
   $\mathbf P$  is pseudo-orthomodular  and 
  both $a$ and $s_{l_S}$ are in $L(s_1', \dots,  s_{l_S-1}')$. Moreover, 
  $LU(a,s_{l_S})\leq LU(X)$ and $LU(a,s_{l_S})\not\leq d_S$ since $a,s_{l_S}\leq LU(X)$ 
  and $a\not\leq d_S$. We conclude that 
  there is an atom $b$ of   $\mathbf P$ such that  $b\in L(U(a,s_{l_S}),s_{l_S}')$, 
  $b\not\leq d_S$, $b\leq LU(X)$, $ b\leq s_1', \dots, b\leq s_{l_S-1}',$ $b\leq s_{l_S}'$, a contradiction 
  with the maximality of $l_S$. Hence  $d_S=LU(X)$ and $\BDM(\mathbf P)$   is  orthomodular. 
  \qed
 \end{proof}
 

 Getting together the previous results we can formulate the following 
 corollary which is a full analogy for  finite pseudo-orthomodular posets to the results 
 on Boolean or relatively pseudo-complemented posets as stated in Theorem \ref{boolpos} 
 or Theorem \ref{relpspos}, respectively. Hence, we conclude
 
 \begin{corollary}\label{finchmhosum}Let $\mathbf P=(P,\leq,{}',0,1)$ be  a 
 finite pseudo-orthomodular poset,
$M(x,y) = L(U(x,y'),y)$  and $R(x,y) = LU(L(x,y),x')$. 
Then $\BDM({\mathbf P})$ is a complete orthomodular lattice.  Moreover, 
 $\BDM(\mathbf P)$ is  a left residuated lattice with respect to $\odot$ and $\to$ reached by the 
 $\BDM$-transformation from $M$ and $R$, respectively.  
\end{corollary}

The next definition and theorem are suggested by a similar result of 
Niederle for Boolean posets (\cite[Theorem 17]{Niederle}).

\begin{definition}\label{doubly} 
Let $\mathbf P=(P,\leq,{}',0,1)$ be a complemented poset.
A  subset  $X$  of  $P$ is  {\em complement-closed and doubly dense in}  
 $\mathbf P$ if the following conditions are satisfied: 
 
 \begin{enumerate}[{\rm(i)}]
\item $(\forall a  \in P) (a =  \bigvee_{\mathbf P}(L(a)  \cap X)  =  %
\bigwedge_{\mathbf P}(U(a) \cap X)$, 
\item $x\in X \implik x'\in X$, 
\item $0, 1\in X$.
\end{enumerate}
\end{definition}

\begin{remark}\label{ddcd}
Recall that any  complement-closed and doubly dense subset $X$ in  $\mathbf P$ is 
a complemented poset with induced order and complementation. Moreover, 
if $\mathbf P=(P,\leq,$ ${}',0,1)$ is a complemented poset then 
$P$ is a  complement-closed and doubly dense subset in its 
Dedekind-MacNeille completion $\BDM(\mathbf P)$. This can be shown by the same 
arguments as in (\cite[Theorem 16]{Niederle}) or can be directly deduced 
from (\cite[Theorem 2.5]{Laren}) so we omit it.
\end{remark}

\begin{theorem}\label{Embnied}
 {\em Embedding  theorem  for finite  pseudo-orthomodular posets.}  \\
  Finite  pseudo-orthomodular posets are precisely  complement-closed 
  and doubly dense subsets of  finite  orthomodular lattices. 
  \end{theorem}
  \begin{proof} \smartqed We have just proved in Corollary \ref{finchmhosum} that every  finite  pseudo-orthomodular posets has  a 
  finite  orthomodular Dedekind-MacNeille completion. Hence it is 
  a  complement-closed and doubly dense subset of a  finite  orthomodular lattice.
  Conversely, let 
  $\mathbf P=(P,\leq,{}',0,1)$ be  a  complement-closed and doubly dense subset of 
  a  finite  orthomodular lattice $(L,\wedge, \vee, {}',0,1)$.  Then 
  $\mathbf P$ is a finite complemented poset. Let us show that $\mathbf P$ is 
   pseudo-orthomodular. Let $x, y\in P$. We can proceed similarly as in Theorem 
   \ref{xTheorem4.5}.   Let $a\in P$.
   We have:
$$
\begin{array}{r c l}
a\in U(x,y) &\ekviv& x, y\leq a  \ekviv x\vee_{\mathbf L}y\leq a%
 \ekviv ((x\vee_{\mathbf L} y)\wedge_{\mathbf L} y')%
\vee_{\mathbf L} y\leq a\\ %
&\ekviv& ((x\vee_{\mathbf L} y)\wedge_{\mathbf L} y') \leq a \text{ and } 
 y\leq a\\
 &\ekviv& \left((\forall z\in P) 
 (z\leq x\vee_{\mathbf L} y \text{ and } z\leq y') \implies  z\leq a\right) \text{ and } 
 a\in U(y)\\
  &\ekviv& \left((\forall z\in P) 
 (z\leq U(x, y) \text{ and } z\in L(y')) \implies  z\leq a\right)  \text{ and } 
 a\in U(y)\\
 &\ekviv& \left((\forall z\in P) 
 (z\in L(U(x, y), y')) \implies  z\leq a\right)  \text{ and } 
 a\in U(y)\\
 &\ekviv& a\in U(L(U(x, y), y'))  \text{ and } 
 a\in U(y)\\
&\ekviv& a\in U(L(U(x,y),y'),y).
\end{array}
$$
We conclude that $U(L(U(x,y),y'),y) = U(x,y)$, i.e., $\mathbf P$ is 
 pseudo-orthomodular.\qed
  \end{proof}

Motivated by a paper \cite{riecanova} we introduce the following definition. 

\begin{definition}\label{xDefinition2.1} {\rm Let $\mathbf P=(P,\leq,{}',0,1)$ be a complemented poset.  
Then $\mathbf P$ is
called {\em strongly $D$-continuous} if and only if for all $B,C\subseteq P$ 
with $B\le C$ the following condition is satisfied:
\begin{enumerate}[{\rm (SDC)}]
\item $\bigwedge\nolimits_{\mathbf P} \{g\in P \mid g\in C \text{ or } g'\in B\}=0$ if and only if 
every lower bound of
$C$ is under every upper bound of $B$.
\end{enumerate}}
\end{definition}

\begin{remark} Recall that the implication: 

\medskip
If $B,C\subseteq P$ 
for a complemented poset  $\mathbf P$ are such that $B\le C$ then 
$\{a\in P\mid a\le C\}\le\{d\in P\mid B\le d\}$ implies that
$\bigwedge\nolimits_{\mathbf P} \{g\in P \mid g\in C \text{ or } g'\in B\}=0$ 
\medskip

\noindent{}from the condition (SDC) is valid in any 
complemented poset  $\mathbf P$ 
since it follows from the fact that the Dedekind-MacNeille completion of 
a  complemented poset is 
always complemented (see \cite[Theorem 2.3., Theorem 2.4.]{Laren}). 
This fact was 
explained and used for Boolean posets in \cite{Halas}. 
\end{remark}

In the following, we establish a characterization of 
 complemented posets with orthomodular 
Dedekind-MacNeille completion. 

\begin{theorem}\label{xTheorem4.4}  Let $\mathbf P=(P,\leq,{}',0,1)$ 
be a complemented poset.  
$\mathbf P$  has an orthomodular 
Dedekind-MacNeille completion if and only if\/ $\mathbf P$  is 
a strongly $D$-continuous pseudo-ortho\-mo\-dular poset.
\end{theorem}
\begin{proof}\smartqed (1) Since $\BDM(\mathbf P)$ is a complemented lattice it is enough 
to check the following condition:
\begin{center}
if $X, Y\in \DM(\mathbf P)$, $X\subseteq Y$ and $X'\wedge Y=0$ then $X=Y$.
\end{center}
We put $B=X$ and $C=U(Y)$. Then $B\leq C$ and 
$\bigwedge\nolimits_{\mathbf P} \{d\in P \mid d\in C \text{ or } d'\in B\}=0$. 
We conclude  from  (SDC) that $X=\bigvee_{\BDM({\mathbf P})} B=\bigwedge_{\BDM({\mathbf P})} C=Y$. 
Hence $\BDM(\mathbf P)$ is orthomodular.

(2) Let $(\DM(\mathbf P),\wedge, \vee,{}',0,1)$ be the orthomodular 
Dedekind-MacNeille completion of  $\mathbf P$. It is enough to verify the following 
implication: 

\medskip
if $B,C\subseteq P$  are such that $B\le C$ then 
$\bigwedge\nolimits_{\mathbf P} \{g\in P \mid g\in C \text{ or } g'\in B\}=0$
 implies that $\{a\in P\mid a\le C\}\le\{d\in P\mid B\le d\}$
\medskip

\noindent{}from the condition (SDC). Let $X=\bigvee_{\BDM({\mathbf P})} B$ and 
$Y=\bigwedge_{\BDM({\mathbf P})} C$. Then $X\subseteq Y$ and $X'\wedge Y=0$. 
Since $\DM(\mathbf P)$ is  orthomodular  we obtain that $X=Y$. We conclude that  
$\{a\in P\mid a\le C\}\le\{d\in P\mid B\le d\}$, i.e., 
$\mathbf P$  is strongly $D$-continuous.
 \qed
\end{proof}

\begin{corollary} \label{corstr}
Every complemented strongly $D$-continuous poset is pseudo-ortho\-mo\-dular. 
Every finite pseudo-ortho\-mo\-dular poset is strongly $D$-continuous.
\end{corollary}

Similarly as for  finite  pseudo-orthomodular posets we have the following 
theorem.

\begin{theorem}\label{AllEmbnied}
 {\em Embedding  theorem  for  strongly $D$-continuous  
 pseudo-orthomodular posets.}  
   Strongly $D$-continuous  pseudo-orthomodular posets are precisely  comple\-ment-closed and doubly dense subsets of complete   orthomodular lattices. 
  \end{theorem}
  \begin{proof} \smartqed From Theorem \ref{xTheorem4.4} we know 
  that every  strongly $D$-continuous pseudo-ortho\-modular poset is 
  a  complement-closed and doubly dense subset in 
  its  orthomodular Dedekind-MacNeille completion. Conversely, let 
  $\mathbf P=(P,\leq,{}',0,1)$ be  a  complement-closed and doubly dense subset of 
  a  complete  orthomodular lattice $(L,\wedge, \vee, {}',0,1)$.  Then 
  $\mathbf P$ is a complemented poset. Let us show that $\mathbf P$ is 
   strongly $D$-continuous. As in Theorem \ref{xTheorem4.4}  it 
   is enough to verify the following 
implication: 

\medskip
if $B,C\subseteq P$  are such that $B\le C$ then 
$\bigwedge\nolimits_{\mathbf P} \{g\in P \mid g\in C \text{ or } g'\in B\}=0$
 implies that $\{a\in P\mid a\le C\}\le\{d\in P\mid B\le d\}$ 
\medskip

\noindent{}from the condition (SDC).  Let $X=\bigvee_{\mathbf L} B$ and 
$Y=\bigwedge_{\mathbf L} C$. Then $X\leq Y$ and $X'\wedge Y=0$ 
(since $u\in P, u\leq X'\wedge Y$ implies $u\leq g$ for all 
$g\in P$ such that  $g\in C \text{ or } g'\in B$, i.e., $u=0$). 
Since $\mathbf L$ is  orthomodular  we obtain that $X=Y$. 
Now, let $a\in P, a\le C$ and $d\in P, B\le d$. Then $a\leq Y=X\leq d$, 
i.e.,  $\mathbf P$ is  strongly $D$-continuous and from Corollary 
\ref{corstr} we have that is also pseudo-ortho\-modular.\qed
  \end{proof}

\begin{acknowledgement}
Support of the research of the first two authors by \"OAD, project CZ~04/2017, and 
of the first author by IGA, project P\v rF~2018~012, and of the second author 
by the Austrian Science Fund (FWF), project I~1923-N25, is gratefully acknowledged.
Research of the third author was supported by the project New approaches to aggregation operators in analysis and processing 
of data, Nr.~18-06915S by Czech Grant Agency (GA\v{C}R). 
\end{acknowledgement}
%

%
%
%

\end{document}